\documentclass[12pt,a4paper]{amsart}
\usepackage[utf8]{inputenc} %

\title{$\R(K_{\aleph_0}, \hat{K}_{2,3})$ is a win for Player 1}
\author[N.~Bowler]{Nathan Bowler}
\author[H.~Ortm{\"u}ller]{Henri Ortm{\"u}ller}
\email{nathan.bowler@uni-hamburg.de, henri.ortmueller@unifr.ch}

\date{\today}

\usepackage{graphicx} %

\usepackage[left=25mm,right=25mm,top=30mm,bottom=30mm,includeheadfoot]{geometry} %
\usepackage{amsmath,amssymb,amsthm} %
\usepackage{thmtools} %
\usepackage{mathtools} %
\usepackage{dsfont} %
\usepackage[colorlinks, linkcolor = magenta]{hyperref} %
\usepackage[nameinlink]{cleveref}
\usepackage{caption} %
\usepackage{subcaption} %
\usepackage[backend=biber]{biblatex} %
\usepackage{tikz, tikz-network} %
\usepackage[many]{tcolorbox} %

\usepackage{float} %
\allowdisplaybreaks %
\usepackage[makeroom]{cancel}

\newcommand{\N}{\mathbb{N}} %
\newcommand{\Z}{\mathbb{Z}} %
\newcommand{\R}{\mathcal{R}} %
\newtheorem{theorem}{Theorem}[section] 
\newtheorem{lemma}[theorem]{Lemma}

\theoremstyle{definition}

\begin{document}

\begin{abstract}
    The \emph{Strong Ramsey game} $\R(B,G)$ is a two player game with players $P_1$ and $P_2$, where $B$ and $G$ are $k$-uniform hypergraphs for some $k \geq 2$. $G$ is always finite, while $B$ may be infinite. $P_1$ and $P_2$ alternately color uncolored edges $e \in B$ in their respective color and $P_1$ begins. Whoever completes a monochromatic copy of $G$ in their own color first, wins the game. If no one claims a monochromatic copy of $G$ in a finite number of moves, the game is declared a draw. For a $t \in \N$, let $\hat{K}_{2,t}$ denote the $K_{2,t}$ together with the edge connecting the two vertices in the partition class of size 2. The purpose of this paper is to give a winning strategy for $P_1$ in the game $\R(K_{\aleph_0}, \hat{K}_{2,3})$.
\end{abstract}

\maketitle

\section{Introduction}
One of the first people studying Strong Ramsey games was Beck. In \cite{BeckRamseyGames2002}, he raised the question, whether $P_1$ has a winning strategy in $\R(K_{\aleph_0}, K_k)$ for all $k \geq 4$ and $\aleph_0$ being the smallest infinite cardinal. Note that $P_2$ does not have a winning strategy in any Strong Ramsey game, due to a strategy stealing argument. For a proof, see for example \cite[Theorem 1.3.1]{MR3524719}. Over the last 20 years, only in the case of $k=4$, Bowler and Gut \cite{bowler2023k4game} proved that $P_1$ has a winning strategy. After proving for a 5-uniform hypergraph, that $P_2$ can reach a draw in corresponding Strong Ramsey game on the complete 5-uniform board, if both players play optimal, Hefetz, Kusch, Narins, Pokrovskiy, Requil\'e and Sarid \cite{MR3645576} suggested to take an intermediary step by asking the following question.
\begin{align*}
    \text{Does there exist a graph $G$, such that $\R(K_{\aleph_0},G)$ is a draw?}
\end{align*}
In 2020, David, Hartarsky and Tiba \cite{MR4073378} improved their result by showing for a $4$-uniform version of $\hat{K}_{2,4}$, that $P_2$ has a drawing strategy in the corresponding Strong Ramsey game. Due to the similarity of their graph and the $\hat{K}_{2,4}$, they raised the question, whether $\R(K_{\aleph_0}, \hat{K}_{2,4})$ is a draw.
Since giving a winning strategy for $P_1$ in $\R(K_{\aleph_0}, \hat{K}_{2,1})$ and $\R(K_{\aleph_0}, \hat{K}_{2,2})$ is straightforward, by proving that $P_1$ has a winning strategy for $\R(K_{\aleph_0}, \hat{K}_{2,3})$, we verify that $\hat{K}_{2,4}$ would be a minimal example, if $\R(K_{\aleph_0},\hat{K}_{2,4})$ was a draw.

\section{Notation}
All graphs considered in this paper are undirected, simple graphs. Hence, we denote an edge $\{x,y\}$ as $xy$ instead for clarity of the presentation.
In the Strong Ramsey game $\R(B,G)$, we will call $B$ the \textit{board} and $G$ the \textit{target graph}. We refer to $P_1$ as \textit{she} and $P_2$ as \textit{he} and we will often say that a player claims, picks or takes an edge instead of colors it in their respective color.
Speaking of colors: For better visualisation, we are going to provide the reader with figures of specific board states. For that we will use the color violet only for edges of $P_1$ and blue only for edges of $P_2$ and in all of those figures, it will be $P_1$'s turn. In order to avoid cases, we do not specify all edges $P_2$ has claimed at a given board state. In the corresponding figure, we denote this by $+k$, if exactly $k$ of $P_2$'s edges are not drawn. If there is an edge $e^*$ with special properties among one of those $k$ edges, we will denote this with $+e^*$ in the figure.
To keep the paper concise, we will not define \emph{strategy} here. For a formal definiton, we refer the reader to \cite[Appendix C]{CombGamesBeck}.

For a given target graph $G$, we say that $P_1$ has a \textit{threat}, if $P_1$ has claimed a copy $H \subseteq B$ of $G-e$ for an edge $e \in E(G)$, such that the corresponding $\hat{e} \in B$ is not yet claimed by either player. We will call $H$ a \textit{threat graph}. For a vertex $x$, $d_{P_1}(x)$ denotes the number of edges $P_1$ has claimed, which are incident to $x$. The definitions above are analogous for $P_2$. If $d_{P_1}(x) = d_{P_2}(x) = 0$, we say the vertex $x$ is \emph{fresh}. In addition, we will refer to the vertices of degree $t+1$ in $\hat{K}_{2,t}$ as \emph{main vertices}.

\section{The proof}

\begin{theorem}
    \label{thm:K23}
    $P_1$ has a winning strategy in $\R(K_{\aleph_0}, \hat{K}_{2,3})$.
\end{theorem}
The rough idea of the proof is to analyze the game in reversed order compared to the order the game is played in. We will begin by proving a lemma, which shows that, if $P_1$ has reached a board state, where she has almost claimed a copy of $\hat{K}_{2,3}$, she can win the game.
\begin{lemma}
    \label{lem:EndPosition}
    In $\R(K_{\aleph_0}, \hat{K}_{2,3})$, suppose $P_1$ has built a $\hat{K}_{2,2}$, such that
    \begin{enumerate}
        \item  $P_2$ does not have a threat;
        \item  $P_2$ has claimed at most 7 edges of the board.
    \end{enumerate}
    If it is $P_1$'s turn, she has a winning strategy.
\end{lemma}
\begin{proof}
    Let $x,y$ be the main vertices of the $\hat{K}_{2,2}$, which $P_1$ has claimed. Observe first that our second condition ensures that at least one of $x, y$ is not a main vertex of a copy of $\hat{K}_{2,2}$, which $P_2$ has claimed. ($x$ and $y$ cannot be main vertices of the same $\hat{K}_{2,2}$, since $xy$ is taken by $P_1$.) Without loss of generality, assume that $P_2$ has not claimed a copy of $\hat{K}_{2,2}$ with main vertex $x$. Then, $P_1$ claims $yz_1$ for a fresh vertex $z_1$.
    \begin{figure}[H]
        \centering
        \includegraphics[width=0.33\textwidth]{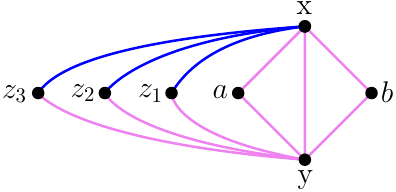}
        \caption{Parts of the board after $P_2$ blocked three times.}
        \label{pic:finalwin}
    \end{figure}
    Since $P_2$ does not have a threat, $P_2$ has to block $P_1$'s threat by taking $xz_1$. Assume for a contradiction that this blockade yields $P_2$ a threat graph $H$. Then $xz_1 \in E(H)$ and since $d_{P_1}(z_1) = 1$ and every edge of $H$ must be incident to a main vertex, we can assume without loss of generality that $x$ and $m$ are the main vertices of $H$ for some vertex $m$. Since $xz_1 \in E(H)$ and $d_{P_1}(z_1) = 1$, the edge that is missing to complete a $\hat{K}_{2,3}$ must be $mz_1$. Hence, $P_2$ already built a $\hat{K}_{2,2}$ with main vertices $x$ and $m$ before $P_1$ took $yz_1$ contradicting our assumption.

    $P_1$ now claims $yz_2$ and $yz_3$ in her next two moves for fresh vertices $z_2$ and $z_3$. $P_2$ has to block the threats of $P_1$ by taking $xz_2$ and $xz_3$, respectively, because he does not create a threat himself as argued above.
    Since $P_2$ does not have a threat after claiming $xz_3$, $P_2$ needs at least two more moves to claim a copy of $\hat{K}_{2,3}$, but it is $P_1$'s turn and with the labeling from \Cref{pic:finalwin}, she wins by claiming two out of $az_1, az_2$ and $az_3$ in her next two moves.
\end{proof}
We are going to trace back all board states, which occur in a game of $\R(K_{\aleph_0}, \hat{K}_{2,3})$ in which $P_1$ is following her strategy to \Cref{lem:EndPosition} beginning with the following.
\begin{lemma}
    \label{lem:triangle}
    If $P_1$ claims a $K_3$ in her first three moves in $\R(K_{\aleph_0}, \hat{K}_{2,3})$, she has a winning strategy.
\end{lemma}
\begin{proof}
    Observe that, $P_1$ can claim a copy of $\hat{K}_{2,2}$ in her next two moves. Since $P_2$ does not have a threat after his fifth move, due to \Cref{lem:EndPosition}, $P_1$ has a winning strategy.
\end{proof}

\begin{lemma}
    \label{lem:Mainlem}
    If $P_1$ has constructed the graph from \autoref{pic:trianglewithedge} in $\R(K_{\aleph_0}, \hat{K}_{2,3})$ after his fourth move, such that there exists an edge $e^*$ of $P_2$, that satisfies either $a \in e^*$ or \linebreak ($e^* \cap \{a,b,c,d\} = \emptyset$ and after his fourth move, we are not in the board state \autoref{pic:SpCase1} or \autoref{pic:SpCase2}), then $P_1$ has a winning strategy.
\end{lemma}
\begin{figure}[H]
    \centering
    \begin{subfigure}[t]{0.19\textwidth}
        \centering
        \includegraphics[height= 20mm, keepaspectratio]{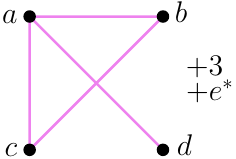}
        \caption{General case}
        \label{pic:trianglewithedge}
    \end{subfigure}
    \hspace{10mm}
    \begin{subfigure}[t]{0.25\textwidth}
        \centering
        \includegraphics[height= 20mm, keepaspectratio]{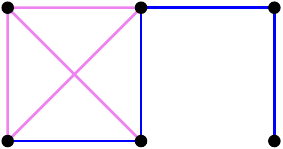}
        \caption{Special case 1}
        \label{pic:SpCase1}
    \end{subfigure}
    \hspace{10mm}
    \begin{subfigure}[t]{0.25\textwidth}
        \centering
        \includegraphics[height= 20mm, keepaspectratio]{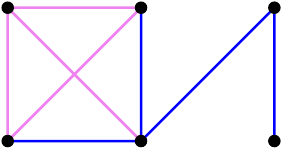}
        \caption{Special case 2}
        \label{pic:SpCase2}
    \end{subfigure}
    \caption{}
\end{figure}
\begin{proof}
    Suppose $P_1$ has claimed the graph from $\autoref{pic:trianglewithedge}$ with the given vertex labeling. If $P_2$ has not taken $bd$ and $cd$ after his fourth move, $P_1$ wins by \Cref{lem:EndPosition} after taking one of them (as in \Cref{lem:triangle}).
    Otherwise, $P_1$ claims $az$ for a fresh vertex $z$ and $zc$ afterwards up to relabeling vertices (\Cref{pic:maincase}). All that is left to show in order to apply \Cref{lem:EndPosition} is that after his sixth move, $P_2$ has no threat.
    \begin{figure}[H]
        \centering
        \includegraphics[width=0.24\textwidth]{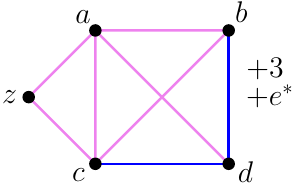}
        \caption{The board after six moves by both players}
        \label{pic:maincase}
    \end{figure}
    Assume for a contradiction that $P_2$ has built a threat graph $H$. Then, all of his six edges must be in $E(H)$. Since $bc$ is taken by $P_1$, $d$ must be a main vertex of $H$. If $a \in e^*$, then $a \in V(H)$, which is not possible, because $ad \in E(P_1)$.

    So suppose that $e^* \cap \{a,b,c,d\} = \emptyset$. Let $e^* = xy$ and without loss of generality, let $x$ be the second main vertex of $H$. Hence $V(H)= \{b,c,d,x,y\}$, so the $\hat{K}_{2,3}$, which $P_2$ threatens to complete is uniquely determined and $E(H)$ must contain three of the four edges $dx, dy, bx, cx$.
    \begin{figure}[H]
        \centering
        \includegraphics[width=0.325\textwidth]{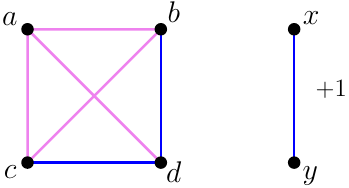}
        \caption{The critical board state after the fourth move of $P_2$.}
        \label{pic:problematicCase}
    \end{figure}
    Therefore, $P_2$ must have claimed $bd, cd, xy$ and one of the edges from $\{dx, dy, bx, cx\}$ after his fourth move. In all four cases, one can check that we obtain the board state from \autoref{pic:SpCase1} or \autoref{pic:SpCase2}. Hence, $P_1$ obtains a winning strategy by \Cref{lem:EndPosition}.
\end{proof}

\begin{lemma}
    \label{lem:SpCase1}
    In the game $\R(K_{\aleph_0}, \hat{K}_{2,3})$, $P_1$ has a winning strategy in the board state from \autoref{pic:SpCase1}.
\end{lemma}
\begin{proof}
    For convenience label the vertices from \Cref{pic:SpCase1} as in \Cref{pic:problematicCase} (where the edge $bx$ is missing). $P_1$ begins by claiming $cx$. If $P_2$ does not pick $ax$, $P_1$ builds a $\hat{K}_{2,2}$ by taking $ax$ (\autoref{pic:SpCase1_1}).
    If $P_2$ took $ax$, $P_1$ would claim $az$ for a fresh vertex $z$ and in her next move, she would claim $zb$ (\autoref{pic:SpCase1_2}), if possible and $zc$ otherwise (\autoref{pic:SpCase1_3}). In all three cases, all that is left to show in order to apply \Cref{lem:EndPosition} is that $P_2$ does not have a threat. So assume for a contradiction that $P_2$ had a threat graph $H$ in one of the three figures below.
    \begin{figure}[H]
        \centering
        \begin{subfigure}[t]{0.29\textwidth}
            \centering
            \includegraphics[height= 27mm, keepaspectratio]{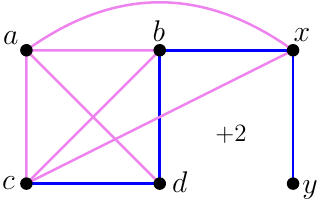}
            \caption{$P_1$ claimed a copy of $\hat{K}_{2,2}$ with main vertices $a$ and $c$.}
            \label{pic:SpCase1_1}
        \end{subfigure}
        \hfill
        \begin{subfigure}[t]{0.31\textwidth}
            \centering
            \includegraphics[height= 27mm, keepaspectratio]{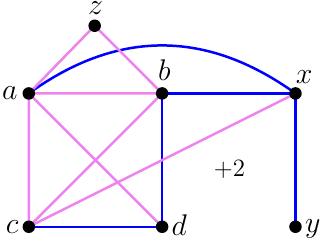}
            \caption{$P_1$ claimed a copy of $\hat{K}_{2,2}$ with main vertices $a$ and $b$.}
            \label{pic:SpCase1_2}
        \end{subfigure}
        \hfill
        \begin{subfigure}[t]{0.32\textwidth}
            \centering
            \includegraphics[height= 27mm, keepaspectratio]{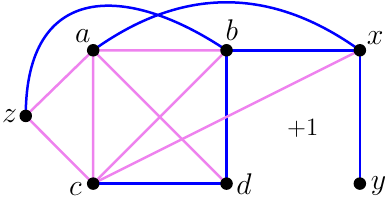}
            \caption{$P_1$ claimed a copy of $\hat{K}_{2,2}$ with main vertices $a$ and $c$.}
            \label{pic:SpCase1_3}
        \end{subfigure}
        \caption{}
        \label{pic:SpCase1Proof}
    \end{figure}
    In \autoref{pic:SpCase1_1}, the only possible choice for main vertices of $H$ would be $d,x$, since every edge has to be incident to a main vertex and all four blue edges must be contained in $E(H)$. But this choice is invalid, because $cx$ is taken by $P_1$. Hence, $P_2$ does not have a threat.

    Let $S \subseteq E(P_2)$ be the set of blue edges drawn in \autoref{pic:SpCase1_2}. We show that every subset $S' \subseteq S$ of four blue edges cannot lie in $H$ by verifying that in each case there is no valid set of main vertices. If $cd \in S'$, then either $c$ or $d$ has to be a main vertex of $H$. Since two edges of $\{ax, bx, xy\}$ must be contained in $S'$, $x$ has to be the other main vertex. But $cx \in E(P_1)$ contradicts that $x$ is a main vertex. If $cd \notin S'$, we either have $b,x$ or $d,x$ as main vertices. $b, x$ cannot be main vertices, since $ab \in E(P_1)$ and $d,x$ cannot be main vertices, since $ad \in E(P_1)$. Hence, $P_2$ does not have a threat in this case.

    In \autoref{pic:SpCase1_3}, $P_2$ does not have a threat, because any set $S$ of five blue edges covers at least six vertices, but $|V(\hat{K}_{2,3})| = 5$, so $S$ cannot be contained in $E(\hat{K}_{2,3})$.
\end{proof}
With those auxiliary lemmas, we are finally able to prove the our main theorem.
\begin{proof}[Proof of \Cref{thm:K23}]
    $P_1$ begins by claiming an edge $ab$. Let $e^*$ be the first edge claimed by $P_2$. Without loss of generality, we can assume that $a \in e^*$ or $e^* \cap \{a,b\} = \emptyset$. $P_1$ takes $ad$ for a fresh vertex $d$. $P_2$ has to claim $bd$ on his next move, because $P_1$ wins by \Cref{lem:triangle} otherwise. $P_1$ takes $ac$ for a fresh vertex $c$.
    \begin{figure}[H]
        \centering
        \includegraphics[width=0.24\textwidth]{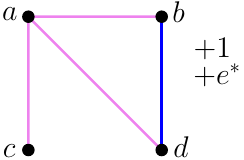}
        \caption{The board state after three moves by each player}
        \label{pic:GC1}
    \end{figure}
    If $a \in e^*$, $P_1$ claims without loss of generality $bc$ in her next move, which is winning for $P_1$ by \Cref{lem:Mainlem}. Therefore, we can assume that $e^* = xy$.
    Almost every possible way of drawing the unspecified edge of $P_2$ in \Cref{pic:GC1} yields a board state in which $P_1$ can force a win by \Cref{lem:Mainlem} after claiming $bc$. The only exceptions are depicted in \Cref{pic:criticalCases} (up to symmetry). 
    \begin{figure}[H]
        \centering
        \begin{subfigure}[t]{0.2\textwidth}
            \centering
            \includegraphics[height= 33mm, keepaspectratio]{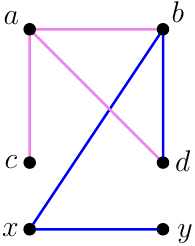}
            \caption{}
            \label{pic:SpCase1MainThm1}
        \end{subfigure}
        \hspace{10mm}
        \centering
        \begin{subfigure}[t]{0.2\textwidth}
            \centering
            \includegraphics[height= 33mm, keepaspectratio]{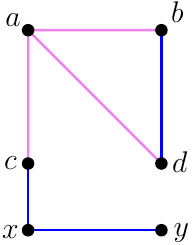}
            \caption{}
            \label{pic:SpCase1MainThm2}
        \end{subfigure}
        \hspace{10mm}
        \centering
        \begin{subfigure}[t]{0.2\textwidth}
            \centering
            \includegraphics[height= 33mm, keepaspectratio]{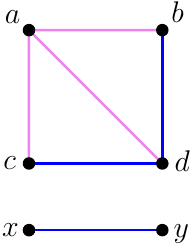}
            \caption{}
            \label{pic:MainThmBadCase}
        \end{subfigure}
        \caption{The critical cases, which can occur on the board (\Cref{pic:GC1}) up to symmetry.}
        \label{pic:criticalCases}
    \end{figure}

    In the cases \hyperref[pic:SpCase1MainThm1]{(a)} and \hyperref[pic:SpCase1MainThm2]{(b)}, $P_1$ takes $bc$. If $P_2$ does not claims $cd$, $P_1$ can force a win by taking $cd$, since she has claimed a copy of $\hat{K}_{2,2}$ in her first five moves (see \Cref{lem:triangle}). If $P_2$ takes $cd$, $P_1$ wins by \Cref{lem:SpCase1}.

    In case \hyperref[pic:MainThmBadCase]{(c)}, $P_1$ claims $az$ for a fresh vertex $z$. If $P_2$ takes $bz$ or $cz$, we can assume that $P_2$ took $bz$ by symmetry, then $P_1$ claims $cz$ (\autoref{pic:Case3_1}). If $P_2$ claims $dz$, $P_1$ claims $cz$ (\autoref{pic:Case3_2}). In all other cases, $P_1$ claims $dz$ in her next move (\autoref{pic:Case3_3}).
    \begin{figure*}[h]
        \centering
        \hspace{15mm}
        \begin{subfigure}[t]{0.21\textwidth}
            \centering
            \includegraphics[height= 40mm, keepaspectratio]{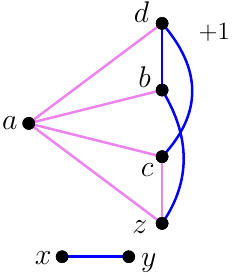}
            \caption{}
            \label{pic:Case3_1}
        \end{subfigure}
        \hfill
        \begin{subfigure}[t]{0.22\textwidth}
            \centering
            \includegraphics[height= 40mm, keepaspectratio]{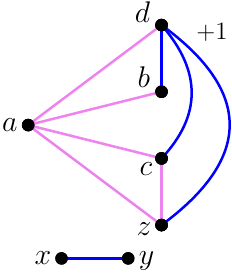}
            \caption{}
            \label{pic:Case3_2}
        \end{subfigure}
        \hfill
        \begin{subfigure}[t]{0.22\textwidth}
            \centering
            \includegraphics[height= 40mm, keepaspectratio]{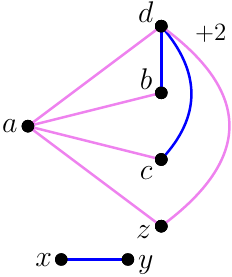}
            \caption{}
            \label{pic:Case3_3}
        \end{subfigure}
        \hspace{15mm}
        \caption{}
        \label{pic:Case3_all}
    \end{figure*}
    In order to complete a $\hat{K}_{2,2}$, $P_1$ claims either $dz$ or $bc$ in \Cref{pic:Case3_1}, $bz$ or $bc$ in \Cref{pic:Case3_2} and without loss of generality $cz$ in \Cref{pic:Case3_3}, since $P_2$ cannot have claimed both $bz$ and $cz$. In all of the cases, $P_1$ has claimed a $\hat{K}_{2,2}$ within her first 6 moves, so in order to apply \Cref{lem:EndPosition}, we will show that $P_2$ has no threat after his sixth move. Assume he had a threat, then all of his edges must be contained in his threat graph.

    In the case of \autoref{pic:Case3_1} and \autoref{pic:Case3_2}, the graph spanned by edges of $P_2$ contains at least $6$ vertices, but $\hat{K}_{2,3}$ has only $5$. Finally, in the case of \autoref{pic:Case3_3}, note that that $P_2$ must have claimed $bz$, otherwise he loses on $P_1$'s next turn, because she has claimed $cz$. Then, $P_2$ does not have a valid choice of main vertices as in \Cref{pic:Case3_1}. Hence $P_2$ does not have a threat in all cases and by \Cref{lem:EndPosition}, this completes the proof.
\end{proof}

\printbibliography

\end{document}